\documentclass[12pt]{amsart}

\usepackage{vmargin}
\setmargrb{1in}{1in}{1in}{1in}
\usepackage{amstext}
\usepackage{amssymb,mathrsfs}
\usepackage{latexsym}
\def\d={\,:=\,}

\font\frakten=eufm10
\newfam\frakfam
\textfont\frakfam=\frakten

\newtheorem{thm}{Theorem}
\newtheorem{lemma}[thm]{Lemma}
\newtheorem{cor}[thm]{Corollary}
\newtheorem{prop}[thm]{Proposition}
\newtheorem{Defn}[thm]{Definition}
\newtheorem{Ex}[thm]{Example}
\newtheorem{Rem}[thm]{Remark}
\newtheorem{Exs}[thm]{Examples}
\newtheorem{Rems}[thm]{Remarks}
\newtheorem{Defrem}[thm]{Definition and Remark}
\newtheorem{Remnt}[thm]{}
\newenvironment{rem}
 {\begin{Rem} \begin{rm}} {\end{rm} \hfill $\Box$ \end{Rem}}

\begin{document}

\author{Hartmut~F{\"u}hr, Jun~Xian}

\address{H.~F{\"u}hr\\Lehrstuhl A f{\"u}r Mathematik
 \\RWTH Aachen University
 \\
D-52056 Aachen\\
Germany.} \email{fuehr@mathA.rwth-aachen.de}

\address{J.~Xian\\Department of Mathematics
 \\Sun
Yat-sen University
 \\
510275 Guangzhou\\
China.} \email{xianjun@mail.sysu.edu.cn}

\title[Relevant sampling in shift-invariant spaces]{Relevant sampling in finitely generated shift-invariant spaces}

\keywords{Random sampling, relevant sampling, shift-invariant
spaces.} \subjclass[2000]{94A20, 42C15, 60E15, 62M30.}

\date{\today}

\begin{abstract}
We consider random sampling in finitely generated shift-invariant
spaces $V(\Phi) \subset {\rm L}^2(\mathbb{R}^n)$ generated by a vector $\Phi =
(\varphi_1,\ldots,\varphi_r) \in {\rm L}^2(\mathbb{R}^n)^r$.
Following the approach introduced by Bass and Gr\"ochenig, we
consider certain relatively compact subsets $V_{R,\delta}(\Phi)$ of
such a space,
 defined in terms of a concentration inequality with respect to a
 cube with side lengths $R$. Under very mild assumptions on the generators, we show that for $R$ sufficiently large, taking $O(R^n log(R^{n^2/\alpha'}))$ many
 random samples (taken independently uniformly distributed within $C_R$) yields
 a sampling set for $V_{R,\delta}(\Phi)$ with high probability. Here $\alpha' \le n$
 is a suitable constant.
 We give explicit estimates of all involved constants in terms of the generators $\varphi_1, \ldots, \varphi_r$.
\end{abstract}

\maketitle
\section{Introduction}
Digital signal processing rests on two basic operations: sampling
and reconstruction. Sampling is the task of transforming the analog
into a digital signal. The converse process is the reconstruction of
the analog signal from the digital signal. However, these problem
cannot be solved without extra information or assumptions on the
analog signal under consideration. There exist several well-understood
ways of formulating these restrictions, eg. by assuming the analog
signal $f$ to be bandlimited, or more generally, that it belongs to
a shift-invariant space \cite{aldroubi2000beurling,
aldroubi2001nonuniform, aldroubi2004, bass2010random, BaGro_12,
eldar2009compressed, xian2014, sun2010, sun2009, xl, xs,
yang2013wei}. Bandlimited signals of finite energy are completely
characterized by their regular samples if they are taken at a
sufficiently high rate (Nyquist criterion), as described by the
famous classical Shannon sampling theorem. A more general class of
such spaces are finitely generated shift-invariant spaces of the
following type
\[\label{1.2}V(\Phi):=\left\{\sum_{k\in\mathbb{Z}^n} C^{T}(k)\Phi(\cdot-k): C\in(\ell^2)^r\right\}\] where
$\Phi=(\varphi_1,\varphi_2,\cdots,\varphi_r)^{T}$ for $\varphi_i\in
L^2(\mathbb R^n)$$ \; (i=1,2,\cdots,r$ and $1\le p\le \infty)$ is
the so-called generator of $V(\Phi)$ and
$C=\{c^1,c^2,\cdots,c^r\}$ with
$\|C\|_{(\ell^2)^r}^2=\sum_{i=1}^r\|c^i\|_{\ell^2}^2$.

In the past years, the random sampling method has been commonly used
in the field of compressed sensing \cite{candes2006robust,
eldar2009compressed} and  image processing\cite{chan2014}.  The
general context of learning from random sampling has been studied by
Cucker, Smale, Zhou, et al. (see \cite{cucker2002smale,
smale2009online}).
%In \cite{Doerr2014}, Doerr obtained the lower
% bound for the discrepancy from the random sampling point set.
Recently, the random sampling problems were studied by Bass and
Gr{\"o}chenig in the multivariate trigonometric polynomials spaces
\cite{bass2005random} and bandlimited functions spaces
\cite{bass2010random, BaGro_12}. Yang and Wei discussed the problem
when some randomly chosen samples $X=\{x_j:j\in J\}$ forms a set of
sampling in the shift-invariant space \cite{yang2013wei}.

Random sampling has become a rather active area of research.
However, so far, most results deal with functions defined on compact domains.
For functions defined on $\mathbb{R}^n$,  however,
one is faced with the dilemma of choosing a proper probability distribution for the sampling set, as there is no uniform distribution on all of $\mathbb{R}^n$.
 As a remedy to this problem, Bass and Gr\"ochenig introduced the notion of {\em relevant sampling} in \cite{bass2010random, BaGro_12}.
 Here, the random sampling sets are confined to a fixed compact subset $K$.
 As a tradeoff, the sampling results are not intended to hold for all elements of the space under consideration, but only for those functions who are concentrated (in a suitable sense) within $K$.

%We say that the shifts of $\Phi$ are stable if there exist positive
%constants $m_p$ and $M_p$ such that for all sequences $C \in
%(\ell^p)^r$, \[\label{2.1}
%m_p\|C\|_{(\ell^p)^r}\le\|\sum_{k\in\mathbb
%Z^d}C(k)^T\Phi(\cdot-k)\|_{L^p(\mathbb R^d)}\le
%M_p\|C\|_{(\ell^p)^r}.\] %For the generator of $V(\Phi)$, we always
%%assume that the shifts of $\Phi$ are stable and denote
%%$\Phi_k=\Phi(\cdot-k).$
%Throughout this paper, we always assume that the shift of $\Phi$ are
%stable and $\varphi_i$ belongs to Wiener-amalgam space $W^1$. For
%$1\le p \neq\infty$, we say a measurable function $f\in W^p$, if it
%satisfies
%$$\|f\|_{W^p}:=\sum_{k\in \mathbb
%Z^d}esssup\{|f(x+k)|^p,x\in[0,1]^d\}\le \infty.$$
% When $p=\infty$
%
%$$\|f\|_{W^p}:=\sup_{k\in\mathbb Z^d}esssup\{|f(x+k)|,x\in[0,1]^d\}\le \infty.$$ Then
%$V^p(\Phi)$ is a space of continuous functions. We denote $(W^p)^r$
%as the vector of space $\Phi=(\varphi_1,\varphi_2,\cdots,\varphi_r)$
%with norm
%$$\|\Phi\|_{(W^p)^r}=\sum_{1\le i\le r}\|\varphi_i\|_{W^p}.$$

Following the approach taken by Bass and Gr\"ochenig
\cite{bass2005random, bass2010random, BaGro_12}, we restrict
attention to the subset
\[\label{2.2}V_{R,\delta}(\Phi):=\left\{f\in V(\Phi):
\int_{C_R}|f(x)|^2dx \ge (1-\delta)\int_{\mathbb R^n}|f(x)|^2dx
\right\},\] where $C_R=[-R/2, R/2]^d$ and $0< \delta <1$. Thus
$V_{R,\delta}(\Phi)$ is the subset of $V(\Phi)$ consisting of those
functions whose energy is largely concentrated on $C_R$.

We are looking for conditions on random sets $X$ satisfying
following inequalities:
\begin{eqnarray}\label{1.1}c\|f\|_{L^2}\leq(\sum\limits_{x_j \in X}|
f(x_j)|^2)^{\frac{1}{2}}\leq C\|f\|_{L^2}.\end{eqnarray}

%It is easy to check that any $f\in V^p(\Phi, C_K, \delta)$ satisfies
%the inequalities (\ref{1.1}) if and only if $f/\|f\|_{L^p(\mathbb
%R^d)}$ does. So we may consider the subset \[\label{2.3}V^p_1(\Phi,
%C_K, \delta):=\{f\in V^p(\Phi, C_K, \delta): \|f\|_{L^p(\mathbb
%R^d)}=1\}.\]
In this paper, we pursue and extend the approach of \cite{BaGro_12} to a rather general setting of finitely generated shift-invariant spaces, with very mild conditions on the generators. While the overall proof strategy could be largely preserved, the details of the arguments in \cite{BaGro_12} often relied on the specific, well-understood setting of prolate spheroidal wave functions, and the adaptation to the general case was not straightforward. In any case, we believe that the subsequent results and arguments provide an interesting contrast and supplement to \cite{BaGro_12}.

The paper is organized as follows. In Section 2,  we introduce the
main result and conditions on the generators. In Section 3, the
localization operator associated to $V(\Phi)$ and $C_R$ is described
and proved. In Section 4, we discuss the random sampling in finite
sums of eigenspaces. At the end, the proof of main result is
presented in Section 5.

\section{Statement of the main result}

Throughout the paper, we consider a finitely generated
shift-invariant subspace $ V(\Phi) \subset {\rm L}^2(\mathbb{R}^n)$,
defined as the closed linear span of a tuple of generators
$(\varphi_1,\ldots, \varphi_r) \in {\rm L}^2(\mathbb{R}^n)^r$,
shifted by the integers. We assume that the associated system $(T_m
\varphi_i)_{i =1, \ldots, r,m \in \mathbb{Z}^d}$ is a frame for
$V(\Phi)$. Furthermore, we fix a dual frame obtained as integer shifts of the vectors
$\widetilde{\varphi}_1,\ldots,\widetilde{\varphi_r} \in {\rm
L}^2(\mathbb{R}^n)$.

Thus we obtain for all $f \in V(\Phi)$, that
\[
 f = \sum_{m,i} \langle f, T_m \widetilde{\varphi}_i \rangle
 T_m\varphi_i
\]
In the following, we will mostly work with the notation
\begin{equation} \label{eqn:p_phi}
 P_\Phi = \sum_{m,i} (T_m \varphi_i) \otimes (T_m \widetilde{\varphi}_i)~,
\end{equation} where unconditional convergence in the strong operator topology
is guaranteed by the frame properties, for all $f \in {\rm L}^2(\mathbb{R}^n)$.
The tensor product notation $v \otimes w$ refers to the rank-one
operator $(v \otimes w) f \mapsto \langle f, w \rangle v$.
Our assumptions show that $P_\Phi$ is the identity on $V(\Phi)$,
and since $(T_k \widetilde{\varphi}_i)_{k,i}$ also span $V(\Phi)$,
the kernel of $P_\Phi$ is $V(\Phi)^\perp$; thus $P_\Phi$ is the orthogonal projection onto $V(\Phi)$.
% Since $P_\Phi$ is selfadjoint, we also get
% \begin{equation} \label{eqn:p_phi_star}
% P_\Phi = \sum_{i,k} (T_k \widetilde{\varphi_i}) \otimes (T_k \varphi_i)
% \end{equation}

We let $\| x\|_{\infty}=\max(|x_1|, \cdots, |x_n|)$. Given $R>0$, we write $Q_R : {\rm L}^2(\mathbb{R}^n) \to {\rm L}^2(\mathbb{R}^n)$ for the orthogonal projection operator $f \mapsto f \cdot \chi_{C_R}$. % Our sampling results will be %formulated for the subset
% \[ V_{R,\delta}(\Phi) = \{ f \in V(\Phi): \| Q_R f \|_2^2 \ge (1-\delta) \| f \|_2^2 \}~. \]

\subsection{Assumptions and chief result}

We collect our assumptions on the generators, with the associated constants, in the following list:
\begin{enumerate}
 \item[(A.0)] {\bf [Bessel Constants]} The upper frame constants for the systems $(T_m \varphi_i)_{m,i}$ and $(T_m \widetilde{\varphi}_i)_{m,i}$ are denoted by $C_0$ and $\widetilde{C}_0$, respectively. Hence, for all $f \in {\rm L}^2(\mathbb{R}^n)$:
\[
  \sum_{m,i} |\langle f, T_m \varphi_i \rangle|^2 \le C_0 \| f \|^2~
 \]
 and
 \[
  \sum_{m,i} |\langle f, T_m \widetilde{\varphi}_i \rangle|^2 \le \widetilde{C}_0 \| f \|^2~.
 \]
 \item[(A.1)] {\bf [Reproducing Kernel]} The point evaluations are bounded linear functionals on $V(\Phi)$. Hence, using the Fischer-Riesz theorem, there exists a family $(v_x)_{x \in \mathbb{R}^d} \subset V(\Phi)$ satisfying $f(x) = \langle f, v_x \rangle$, for all $f \in V(\Phi)$. We assume that
 \[ C_1 = C_1(\Phi) = \sup_x \| v_x \|_2 < \infty ~. \] This implies in particular $\| f \|_\infty \le C_1 \| f \|_2$, for all $f \in V(\Phi)$.
 \item[(A.2)]{\bf [Plancherel-Polya-type inequality]} There exists a constant $C_2= C_2(\Phi)$ such that for every subset $\Gamma \subset\mathbb{R}^n$ with covering index
 \[
  N_0(\Gamma) = \max_{k \in \mathbb{Z}^n} ~{\rm card} \left( \Gamma \cap (k + [-1/2,1/2]^n) \right)
 \]
 and every $f \in V(\Phi)$, we have
 \[ \sum_{\gamma \in \Gamma} |f(\gamma)|^2 \le C_2 N_0(\Gamma) \| f \|_2^2 ~.\]
  \item[(A.3)]{\bf [Decay property]} There exists $\alpha>0$ and $C_3 = C_3(\Phi)$, such that for all $i = 1, \ldots, k$: $\| \varphi_i \cdot (1- \chi_{C_R})\|_2^2 \le C_3 R^{-\alpha}$.
\end{enumerate}

Clearly, the decay property is fulfilled by any vector of compactly supported functions.

\begin{rem} \label{rem:bandlimited}
Observe that the requirements are rather mild. They are in fact
fulfilled by the $n$-dimensional sinc function $\varphi$ and the
space $V(\varphi)$ of bandlimited functions: The shifts of the sinc
function provide an ONB, in particular we may take
$\widetilde{\varphi} = \varphi$. We therefore get $C_0 =
\widetilde{C_0} =1$. In addition, the sinc function acts as a
reproducing convolution kernel for $V(\varphi)$, thus (A.1) holds
with $C_1 = 1$. The Plancherel-Polya constant for $V(\varphi)$ was
explicitly computed as $C_2 = e^{n\pi}$ in the appendix of
\cite{BaGro_12}. The localization property holds with $C_3 = n$ and
$\alpha = 1$. Hence the following result indeed provides a
generalization of the main theorem \cite{BaGro_12}, although with
less sharp constants, and a slightly worse sampling rate: Instead of
$O(R^n \log(R^n))$, we obtain $O(R^n\log(R^{n^2/\alpha'}))$.
\end{rem}

Throughout this paper, we will repeatedly refer to the constants $\alpha' = \min(n,\alpha)$, for $\alpha$ from assumption (A.3), and
\begin{equation}
\label{eqn:beta}
 \beta = 3 +2 \sqrt[\alpha]{2^{n+2} r C_0 \widetilde{C}_0^2 C_3}\;.
\end{equation}

The main result of this paper is the following generalization of \cite[Theorem 1]{BaGro_12}:
\begin{thm}
 \label{thm:main}
  Assume that the frame generators fulfill assumptions (A.0)-(A.3).
 Let $(x_j)_{j \in \mathbb{N}}$ denote a sequence of independent random variables, each uniformly distributed in $C_R$.
 Let
 \[
 ~ R_0 = \max\left(1,\sqrt[\alpha]{2 C_3},\sqrt[n]{C_1^2} \right)~.
 \]
Let $R \ge R_0$, and assume that $\delta, \nu \in (0,1)$ are sufficiently small to guarantee that
\[\label{3equ}
 \frac{\nu^2}{C_1^2 (1+\nu/3)} \le 3 \log 3 -2  \mbox{ and } A = \frac{1}{2} - \delta - \nu - 12 \delta C_2 > 0~.
\]
Let $0 < \epsilon < 1$. If the number $s$ of samples satisfies
\begin{equation} \label{eqn:samp_dens}
 s \ge R^n \frac{1+\nu/3}{\nu^2} \log \frac{2 \beta^n R^{n^2/\alpha'}}{\epsilon}~,
\end{equation} then the sampling inequality
\begin{equation}
 \label{eqn:samp_ineq_main}
\forall f \in V_{R,\delta}(\Phi) ~:~ A s R^{-n} \| f \|_2^2 \le \sum_{j=1}^s |f(x_j)|^2 \le s \| f \|_2^2
\end{equation} holds with probability at least $1-\epsilon$.
\end{thm}

\subsection{Generators fulfilling the assumptions of Theorem \ref{thm:main}}

As will be seen shortly, our main result applies to large classes of generators.
For the formulation of the following result, recall the definition of the Wiener amalgam spaces:
Given a function $f: \mathbb{R}^n  \to \mathbb{C}$, we define its Wiener amalgam norm, for $1 \le p < \infty$, via
\[
\| f \|_{W(L^p)}^p = \sum_{k \in \mathbb{Z}^n} {\rm ess\; sup}_{x
\in [0,1]^n} |f(x+k)|^p\;\;,
\]
and denote the space of all continuous $f$ for which this norm is
finite by $W_0(L^p)$. It has been noted in
\cite{aldroubi2001nonuniform} that Wiener amalgam norms are useful
tools for the study of sampling problems, and the following results
provide further evidence for this principle.

\begin{prop}
Assume that $\Phi \in (W_0(L^1))^r$. Then assumptions (A.1) and
(A.2) are fulfilled.
\end{prop}

\begin{proof}
At least formally, the reproducing kernel can be obtained in a straightforward way from (\ref{eqn:p_phi}):
For all $f \in V(\Phi)$, we have
\begin{eqnarray*}
 f(x) & = & \sum_{k,i} \langle f, T_m \widetilde{\varphi}_i \rangle T_m \varphi(x) \\
  & = & \sum_{k,i} \int_{\mathbb{R}^n} f(y) \overline{\widetilde{\varphi}_i(y-k)} dy \varphi_i(x-k) \\
  & = & \int_{\mathbb{R}^n} f(y) \overline{v_x(y)} dy~,
\end{eqnarray*}
where
\[
 v_x (y) = \sum_{k,i} \widetilde{\varphi}_i(y-k) \overline{\varphi_i(x-k)}~.
\]
Now \cite{aldroubi2004} establishes that the sum actually converges
and yields a square-integrable $v_x$. %(Note that the statement there
%was only formulated for the single generator case, but the argument
%carries over immediately to the general case.)
Furthermore, the fact
that the translates of $\Phi$ are Bessel sequence allow to conclude
that
\begin{eqnarray*}
\sup_{x \in \mathbb{R}^n} \| v_x \|_2^2 & \le &  C_0 \sup_{x \in \mathbb{R}^n} \| (\varphi_i(x-k))_{i,k}\|_{\ell^2(\mathbb{Z}^n)^r}^2 \\
& = & C_0 \sup_{x \in [0,1]^n} \sum_{i,k} | \varphi_i(x-k)|^2 \\
& \le & C_0 \sum_{i,k}  \sup_{x \in [0,1]^n}  | \varphi_i(x-k)|^2 \\
& = & C_0 \sum_{i} \| \varphi_i \|_{W(L^2)}^2 \\
& \le & C_0 \sum_{i} \| \varphi_i \|_{W(L^1)}^2 < \infty,
\end{eqnarray*}
using the norm-decreasing inclusion $W_0(L^1) \subset W_0(L^2)$.
This yields assumption (A.1) with constant
\[
C_1 = \left( C_0 \sum_{i=1}^r \| \varphi_i \|_1^2 \right)^{1/2}\;\;.
\]

For the proof of the Plancherel-Polya inequality, first note that it
is sufficient to consider sets $\Gamma \subset \mathbb{N}$ with $N_0(\Gamma) =1$; the more general statement then follows from writing a set $\Gamma$ with $N_0(\Gamma) = K$ as the union of $K$ sets with density one.

 Given $f \in V(\Phi)$, define ${\rm osc} (f): \mathbb{R}^n \to \mathbb{R}_0^+$ as
 \[
  {\rm osc} (f)(x) = \sup_{\| y \|_\infty \le 1/2} | f(x)-f(x+y)| ~.
 \]
 By \cite{aldroubi2004} there exists a constant $M>0$ such that
\[
  \forall f \in V(\Phi)~:~ \|{\rm osc}  f \|_2 \le M \| f \|_2~.
 \] %Again, the referenced statement is formulated only for the single-generator case, but extends verbatim to the general setting.

Now let $\Lambda = \{ k \in \mathbb{Z}^n: \Gamma \cap k + [-1/2,1/2)^n \not= \emptyset \}$. By essential disjointness of the shifted cubes, we have that
 \[
  \sum_{k \in \Lambda} \int_{k +  [-1/2,1/2]^n} |f(x)|^2 dx \le \| f \|_2^2~.
 \]
We can relate this sum to $\sum_{\gamma \in \Gamma} |f(\gamma)|^2$ as follows:
 For each $\gamma \in \Gamma$ pick a $k_\gamma \in \Lambda$ such that $\gamma \in k_\gamma +  [-1/2,1/2)^n$. $k_\gamma$ is uniquely determined, and by assumption on $\Gamma$, we get that $\Gamma \ni \gamma \mapsto k_\gamma$ is one-to-one. We then get the following series of estimates:
 \begin{eqnarray*}
  \sum_{\gamma \in \Gamma} \left| |f(\gamma)|^2 - \int_{k_\gamma +  [-1/2,1/2]^n} |f(x)|^2 dx \right|& \le &
  \sum_{\gamma \in \Gamma} \left| \int_{k_{\gamma} + [-1/2,1/2]^n} |f(\gamma)|^2-|f(x)|^2 dx  \right| \\
  & \le & \sum_{\gamma \in \Gamma} \int_{k_\gamma + [-1/2,1/2]^n} \left|  |f(\gamma)|^2-|f(x)|^2 \right| dx \\
  & \le & \sum_{\gamma \in \Gamma} \int_{k_\gamma + [-1/2,1/2]^n} |f(\gamma) - f(x)| ~(|f(\gamma)|+|f(x)| )dx \\
  & \le & \int_{\mathbb{R}^n} {\rm osc} f(x) \left( 2|f(x)| + {\rm osc} f(x) \right) dx \\
  & \le & \| {\rm osc}f \|_2 \left( 2 \| f \|_2 + \| {\rm osc} f \|_2 \right) \\
  & \le & M(M+2) \|f \|_2^2~
 \end{eqnarray*}
But this implies
\[
 \sum_{\gamma \in \Gamma}  |f(\gamma)|^2 \le M(M+2) \| f \|_2^2 +  \sum_{k \in \Lambda} \int_{k +  [-1/2,1/2]^n} |f(x)|^2 dx \le (M+1)^2 \| f \|_2^2 ~,
\] and the Plancherel-Polya inequality is established.
\end{proof}

We thus obtain easily checked criteria in terms of continuity and moderate decay:
\begin{cor}
Assume that $\Phi$ is vector of functions generating a frame under shifts, and consisting of continuous functions satisfying the decay estimate
\[
\forall i=1,\ldots, r~:~|\varphi_i(x)| \le C (1+\|x \|_{\infty})^{-n-\epsilon} \;\;,
\] for some $\epsilon>0$. Then conditions (A.1) through (A.3) are fulfilled, with $\alpha=n$.
\end{cor}
\begin{proof}
The decay estimate implies $\Phi \in (W_0(L^1))^r$, and thus (A.1)
and (A.2) follow from the previous proposition. (A.3), with $\alpha
= n$, is easily verified.
\end{proof}

\section{The localization operator associated to $V(\Phi)$ and $C_R$}

Throughout the remainder of the paper, we will always assume that $\Phi$ is a set of frame generators fulfilling assumptions (A.0) through (A.3).
The proof strategy for our main result is an adaptation of the method devised by Bass and Gr\"ochenig \cite{BaGro_12} for the special case of bandlimited functions. We introduce the localization operator
\[
 A_ R = P_\Phi \circ Q_R \circ P_\Phi~.
\] We will show that $A_R$ is a selfadjoint Hilbert-Schmidt operator, and therefore has a basis of eigenvectors with associated square-summable spectrum. Denoting by $\mathcal{P}_N$ the projection onto the span of the eigenvectors associated to the largest $N$ eigenvalues, we then establish a random sampling theorem for this space. We then show how sampling of the elements in $V_{R,\delta}(\Phi)$ can be related to sampling in $\mathcal{P}_N$. The proper choice of $N$, which will allow to transfer the random sampling result to $V_{R,\delta}(\Phi)$, depends on certain estimates concerning the decay of the spectrum of $A_R$.

Note that the following result is valid for all shift-generated frames of closed subspaces of ${\rm L}^2(\mathbb{R}^n)$, without any further assumptions on the generators.
\begin{lemma} \label{lem:AR_HS}
 $A_R$ is a positive-semidefinite Hilbert-Schmidt operator.
\end{lemma}
\begin{proof}
Positive semidefiniteness of $A_R$ follows from
\[ \langle A_R f, f \rangle = \langle Q_R P_\Phi f, P_\Phi f \rangle = \| Q_R P_\Phi f \|_2^2 ~,\]
using that $Q_R$ is a selfadjoint projection.

We next show that $Q_R P_\Phi$ is Hilbert-Schmidt, which will imply that $A_R$ is Hilbert-Schmidt, as a composition of a bounded and a Hilbert-Schmidt operator.

For $(j,m) \in \{ 1, \ldots,k \} \times \mathbb{Z}^n$, we have
\[
 Q_R \circ (T_m \varphi_j) \otimes (T_m \widetilde{\varphi}_j)   = ( Q_R T_m \varphi_j)  \otimes (T_m \widetilde{\varphi}_j) ~.
\]
 Using (\ref{eqn:p_phi}) and boundedness of $Q_R$, we can write
 \[
  Q_R P_\Phi = \sum_{j, m} (Q_R T_m  \varphi_j) \otimes (T_m \widetilde{\varphi}_j)
 \] with unconditional convergence in the strong operator topology.

 We next prove that
 \begin{equation} \label{eqn:normsum}
  \sum_{j,m} \| Q_R T_m \varphi_j \|_2^2 < \infty~.
 \end{equation}
 For fixed $k \in \mathbb{Z}$ with $k >R$ and $m_1,m_2 \in k \mathbb{Z}^n$ with $m_1 \not= m_2$, the sets $C_R+m_1$ and $C_R+m_2$ are disjoint. It follows for arbitrary $x \in \mathbb{R}^d$
 \begin{eqnarray*}
  \sum_{m \in \mathbb{Z}^n} \| Q_R T_{k m+x} \varphi_j\|_2^2 & = & \sum_{m \in \mathbb{Z}} \int_{C_R} |\varphi_j(y-x-k m)|^2 dy \\
  & = & \sum_{m \in \mathbb{Z}^n} \int_{C_R+k m} |\varphi_j(y-x)|^2 dy \le \| \varphi_{j} \|_2^2~,
 \end{eqnarray*}
 by the above observed disjointness. Since $\mathbb{Z}^n$ can be covered by $k^n$ cosets of $k \mathbb{Z}^n$, we obtain
 \[
   \sum_{m,j} \| Q_R T_m \varphi_j \|_2^2 < \infty~.
 \]

Thus, with $\vartheta_{m,j} = Q_R T_m \varphi_j$, we have
\[
 Q_R P_\Phi = \sum_{m,j} \vartheta_{m,j} \otimes  (T_m\widetilde{\varphi}_j)~,~\sum_{m,j} \| \vartheta_{m,j}\|_2^2 < \infty~,
\]
and hence
\begin{equation} \label{eqn:Q_abs}
 (Q_R P_\Phi)(Q_R P_\Phi)^* = \sum_{m,j} \left( Q_R P_\Phi T_m \widetilde{\varphi}_j \right) \otimes \vartheta_{m,j} =  \sum_{m,j} \left( Q_R T_m \widetilde{\varphi}_j \right) \otimes \vartheta_{m,j}
\end{equation} with convergence in the strong operator topology. By the same argument as for equation (\ref{eqn:normsum}), we see that
\[
 \sum_{m,j} \| Q_R T_m \widetilde{\varphi}_j \|_2^2 < \infty~.
\] This observation, in combination with (\ref{eqn:normsum}) and the Cauchy-Schwarz inequality, yields that
\[
\sum_{m,j} {\rm trace}\left( \left| \left( Q_R T_m \widetilde{\varphi}_j \right) \otimes \vartheta_{m,j}\right| \right) = \sum_{m,j} \left\| Q_R T_m \widetilde{\varphi}_j \right\| \; \left\|\vartheta_{m,j} \right\|  < \infty\;.
\] Hence we see that the expansion (\ref{eqn:Q_abs}) in fact converges in the trace class norm as well, finally implying that $Q_R P_\Phi$ is Hilbert-Schmidt.
\end{proof}

It follows that $A_R$ has an ONB of eigenvectors. Denoting the nonzero eigenvalues by $(\lambda_n)_{n \in I}$ and the associated eigenfunctions by $(\psi_n)_{n \in I}$, with $I$ being either $\mathbb{N}$ or $\{1, \ldots, M\}$ for some integer $M$, $A_R$ is given as the sum
\[
 \sum_{n \in I} \lambda_n \psi_n \otimes \psi_n~.
\]
As the equation $\psi_n = \lambda_n^{-1} P_\Phi^* Q_R P_\Phi \psi_n$ shows, we have $\psi_n \in V(\Phi)$, and thus $P_\Phi(\psi_n) = \psi_n$.
 Since $A_R$ is Hilbert-Schmidt, we have $\sum_{n} |\lambda_n|^2 < \infty$, and since it is positive-semidefinite, we may assume $\lambda_1 \ge \lambda_2 \ge ... \ge 0$. Furthermore, since $A_R$ is a composition of projections, we have $\lambda_1 \le \| A_R \|_{op} \le 1$.  We let
 \[ \mathcal{P}_N = {\rm span}\{ \psi_n : n = 1, \ldots, N \} ~.\]

 For the space of bandlimited functions, the eigenfunctions are the well-known prolate spheroidal wave functions introduced by Slepian and Pollak \cite{SlPo_61}.  For the sampling results derived below, some information on the spectrum of $A_R$ is necessary. We let
 \[
  N(R) = \max \{ n \in \mathbb{N}~: ~\lambda_n \ge 1/2 \}~,
 \] and $N(R)=0$ whenever $\lambda_1 < 1/2$. Thus, whenever $N(R)>0$, then $\lambda_{N(R)} \ge 1/2 > \lambda_{N(R)+1}$.

 The following lemma provides an estimate for $N(R)$, derived from the decay assumptions on the generators.

\begin{lemma} \label{lem:decay_eigenvals}
Let $\beta$ be defined by (\ref{eqn:beta}).
Then for all $R>\max(1,\sqrt[\alpha]{2 C_3})$, the inequalities $0 < N(R) \le \beta^n R^{n^2/\alpha'}$ hold.
\end{lemma}
\begin{proof}
 We use the well-known minimax formula
 \[
  \lambda_m = \inf \{ \sup \{ \langle A_R f, f\rangle~:~ f \bot \mathcal{H}, \| f \|_2 = 1 \} : \mathcal{H} \subset {\rm L}^2(\mathbb{R}^n),  \dim(\mathcal{H}) \le m \}~.
 \]
Now fix $S>R$, and consider
\[
 \mathcal{H}_S = {\rm span} \{ T_m \widetilde{\varphi}_i : 1 \le i \le r~, \| m \|_\infty \le S/2 \}~.
\] It follows that ${\rm dim}(\mathcal{H}_S) = (2 \lfloor S/2 \rfloor +1)^n \le (\lfloor S \rfloor +1)^n$.
Next assume that $f \in \mathcal{H}_S^\bot$ is a unit vector.
Then we obtain
\begin{eqnarray*}
 \langle A_R f, f \rangle = \sum_{m,j,m',j'} c_{m,j} a_{m,j,m',j'} \overline{ c_{m',j'}}~,
\end{eqnarray*}
with $c_{m,i} = \langle f, T_m \widetilde{\varphi}_i \rangle$, and an infinite, positive semidefinite matrix
\[ \mathcal{A} = (a_{m,j,m',j'})_{((m,j),(m',j')) \in (\mathbb{Z}^n \times \{ 1, \ldots, r\})^2} \] defined by
\[
 a_{m,j,m',j'} = \left\{ \begin{array}{cc} \langle Q_R T_m \varphi_j, T_{m'} \varphi_{j'} \rangle  & \min(\| m \|_\infty,\|m' \|_\infty) > S/2 \\  0 & \mbox{otherwise} \end{array} \right.~,
\] recall the assumption $f \bot T_m \widetilde{\varphi}_i$, for $\| m \|_\infty \le S/2$.
In particular, we get
\[
 \langle A_R f, f \rangle \le \| c \|_2^2 \| \mathcal{A} \|_{\rm op} \le \widetilde{C}_0 \| \mathcal{A} \|_{\rm op}
\]
We estimate the Hilbert-Schmidt norm of the matrix, as follows:
\begin{eqnarray*}
\| \mathcal{A} \|_{HS}^2 & = &
 \sum_{m,j,m',j',\|m\|_\infty \ge S, \|m' \|_\infty > S/2} |a_{m,j,m',j'}|^2 \\
 & \le & \sum_{m,j,\|m \|_\infty > S/2} \sum_{m',j'} \left| \langle Q_R T_m \varphi_j, T_{m'} \varphi_{j'} \rangle \right|^2 \\
 & \le & \sum_{m,j, \|m \|_\infty > S/2 S} C_0 \| Q_R T_m \varphi_j \|_2^2 ~.
\end{eqnarray*}
We can now employ a similar reasoning as in the proof of Lemma \ref{lem:AR_HS}: Picking $k = \lfloor R \rfloor +1$, we have for arbitrary distinct $\ell,\ell' \in \{ 0, \ldots, k-1 \}^n$ that $\ell + C_R \cap \ell'+C_R$ has measure zero. Furthermore, $m + C_R \cap C_{S/2-R}$ has measure zero, whenever $\| m \|_\infty> S/2$. This implies that
\begin{eqnarray*}
 \sum_{m \in \mathbb{Z}^n, \|m \|_\infty>S/2} \| Q_R T_m \varphi_j \|_2^2 & = & \sum_{\ell \in \{ 0, \ldots,k-1\}^n}
 \sum_{m \in \ell + k \mathbb{Z}^n,  \|m \|_\infty>S/2} \| Q_R T_m \varphi_j \|_2^2
 \\
 & = & \sum_{\ell \in \{ 0, \ldots,k-1 \}}^n
 \sum_{m \in \ell + k\mathbb{Z}^n,  \|m \|_\infty>S/2} \int_{C_R + m} |\varphi_j(x)|^2 dx \\
 & \le & \sum_{\ell \in  \{ 0, \ldots,k-1\}^n} \int_{\mathbb{R}^n \setminus C_{S/2-R}} |\varphi_j(x)|^2 dx \\
 & \le & C_3 \left( \lfloor R \rfloor +1\right)^n (S/2-R)^{-\alpha}~,
\end{eqnarray*}
using the decay assumption (A.3).
Thus we arrive at
\[
 \langle A_R f,f \rangle^2 \le \widetilde{C}_0^2 \| \mathcal{A} \|_{\rm op}^2 \le  \widetilde{C}_0^2 \| \mathcal{A} \|_{\rm HS}^2 \le r \widetilde{C}_0^2 C_0 C_3  \left( \lfloor R \rfloor +1\right)^n (S/2-R)^{-\alpha}
\]
For $S = (\beta-1) R^{n/\alpha'}$ with $\beta$ according to (\ref{eqn:beta}), one finds that $R>1$ and the definition of $\alpha'$ yield that $R \le R^{n/\alpha'}$, and hence
\begin{eqnarray*}
  r \widetilde{C}_0^2 C_0 C_3 (S/2 - R)^{-\alpha} (\lfloor R \rfloor +1)^n & \le &  r \widetilde{C}_0^2 C_0 C_3 \left(\frac{\beta-3}{2} \right)^{-\alpha} R^{-n} (\lfloor R \rfloor +1)^n \\
  & \le &  r \widetilde{C}_0^2 C_0 C_3 2^n \left( \frac{\beta-3}{2} \right)^{-\alpha} \\
  & \le & \frac{1}{4}~,
\end{eqnarray*} by definition of $\beta$. Hence we have shown
\[
 \langle A_R f,f \rangle^2 < 1/4~
\] for all unit vectors $f \in \mathcal{H}_S^\bot$.

If $R>1$, then $N = (\lfloor (\beta -1) R^{n/\alpha'} \rfloor + 1)^n \le \beta^n R^{n^2/\alpha'}$, and the minimax estimate yields $\lambda_N < 1/2$.
On the other hand, (A.3) implies
\[
 \lambda_1 \ge \langle A_R \varphi_i, \varphi_i \rangle \ge 1- C_3 R^{-\alpha} > 1/2
\] as soon as $R>\sqrt[\alpha]{2 C_3}$.
Hence we find for all $R> \max(1,\sqrt[\alpha]{2 C_3})$ that
\[ 0 < N(R) \le \beta^n R^{n^2/\alpha'} ~.\]
\end{proof}

 \begin{rem}
  The proof of Lemma \ref{lem:decay_eigenvals} is the only place in the paper where we employ
  the decay assumption (A.3) on the generators. All subsequent estimates of $N(R)$ in the
  following depend on this result. The case of bandlimited functions provides one example
  where similar or sharper estimates may be available by alternative methods;
  here the estimate $\lfloor R \rfloor -1 \le N(R) \le \lfloor R \rfloor +1$ can be
  shown by Fourier-analytic arguments, see \cite{LaPo_62} for the one-dimensional case.
   This is the main reason for the suboptimal sampling rate $O(R^n \log R^{n^2/\alpha'})$ stated
    in Remark \ref{rem:bandlimited} for the bandlimited case. Using the estimate from \cite{LaPo_62}
    instead of Lemma \ref{lem:decay_eigenvals}, our subsequent arguments provide
    a sampling rate of $O(R^n \log R^{n})$, just as in \cite{BaGro_12}.
 \end{rem}

\section{Random sampling in finite sums of eigenspaces}

\label{sect:rs_bandlimited}

We continue our adaptation of \cite{BaGro_12}. Recall from the previous
 section that $\lambda_1 \ge \lambda_2 \ge ... $ are the eigenvalues of $A_R$,
 with corresponding eigenfunctions $\psi_1,\psi_2,\ldots$. The 
 span of the first $N$ eigenfunctions is denoted by $\mathcal{P}_N$.
 We let $\Delta_N = {\rm diag}(\lambda_1,\ldots,\lambda_N)$.

The aim of this section is to prove a random sampling statement for
$\mathcal{P}_N$. It will follow by applying a matrix Bernstein
inequality (stated in the following Theorem \ref{bernstein}), which
uses the following notation:
 For $A \in
\mathbb{C}^{N \times N}$, we let $\| A \|$ denote the operator norm
with respect to the euclidean norm. Further, the inequality $A \le
B$ for two matrices $A,B$ of equal size means that $B-A$ is positive
semidefinite.

\begin{thm}{\cite{Ma_2014}}\label{bernstein}
Let $X_j$ be a sequence of independent, random self-adjoint $N\times
N$-matrices. Suppose that $\mathbb{E}X_j =0$ and $\| X_j\|\leq B$
a.s. And let $\sigma^2=\| \sum\limits_{i=1}^{s}
\mathbb{E}(X_j^2)\|$. Then for all $t>0$,
\[\mathbb{P}\big(\lambda_{max}\big( \sum\limits_{i=1}^{s}X_j\big)\geq t\big)\leq N exp\big(-\frac{t^2/2}{\sigma^2 + Bt/3}\big)\]
where $\lambda_{max}(U)$ is the largest singular value of a matrix
$U$ so that $\|U\|=\lambda_{max}(A^*A)^{1/2}$ is the operator norm.
\end{thm}

The random matrices under consideration are constructed as follows:
For each $j \in \mathbb{N}$ and $\ell \in \{ 1, \ldots, k \}$, we
introduce the $N \times N$ rank-one random matrix $T^\ell_j$ defined
by
 \begin{equation} \label{eqn:T_j}
  (T_j)_{k,l} = \psi_k(x_j) \overline{\psi_l(x_j)}~.
 \end{equation} Here the $x_j$ denote i.i.d. random variables, uniformly distributed on $C_R$.
We finally let
\begin{equation} \label{eqn:X_j}
X_j = T_j - \mathbb{E}(T_j)\;.
\end{equation}

The following provides useful estimates for the constants in the
matrix Bernstein inequality. It is an analog of \cite[Lemma
4]{BaGro_12}.
\begin{lemma} \label{lem:exp_T_j}
 Let $X_j$ be defined via (\ref{eqn:T_j}) and (\ref{eqn:X_j}). Then the following hold:
\begin{eqnarray*}
 \mathbb{E}(X_j) = 0 & , & \| X_j \| \le \max(C_1^2, R^{-n}) ~~ (a.s.), \\
 \mathbb{E}(X_j^2) \le R^{-n} C_1^2 \Delta_N & , & \sigma^2 = \left\| \sum_{j=1}^s \mathbb{E}(X_j)^2 \right\| \le s C_1^2 R^{-n} ~.
\end{eqnarray*}
\end{lemma}
\begin{proof}
It is obvious that $\mathbb{E}(X_j) = 0$. Since both $T_j$ and
$\mathbb{E}(T_j)$ are positive semidefinite, we have
\[
 \| X_j \| \le \max( \| T_j \|, \| \mathbb{E} T_j \| ) \le \max (\| T_j \|, R^{-n})~.
\] Furthermore, recall from (\ref{eqn:fxj}) that
\[
 \| T_j \| = \sup_{\| c \|_2 = 1} |\langle c, T_j c \rangle|= \sup_{f \in \mathcal{P}_N, \| f \|_2 = 1} |f(x_j)|^2 \le C_1^2
\]using assumption (A.1) and $\mathcal{P}_N \subset V(\Phi)$.

We next compute
\[
 \mathbb{E}(X_j^2) = \mathbb{E}(T_j^2) - \left( \mathbb{E}(T_j) \right)^2 = \mathbb{E}(T_j^2) - R^{-2n} \Delta_N^2.
\] The square of the rank-one matrix $T_j$ is computed as
\[
 T_j^2 = \left( \sum_{\ell=1}^N |\psi_\ell(x_j)|^2 \right) T_j ~.
\] Using the reproducing kernel $(v_x)_{x \in \mathbb{R}^n}$ for $V(\Phi)$, together with the fact that the eigenfunctions are an orthonormal system, we can estimate
\[
 \sum_{\ell=1}^N |\psi_\ell(x_j)|^2 = \sum_{\ell=1}^N |\langle \psi_\ell, v_{x_j} \rangle|^2 \le
 \sum_{\ell=1}^\infty |\langle \psi_\ell, v_{x_j} \rangle|^2 \le \| v_{x_j} \|_2^2 \le C_1^2 ~.
\] Thus
\[
 T_j^2 \le C_1^2 T_j~,
\] and since the expected value of a positive-semidefinite matrix valued random variable is positive semidefinite, we obtain
\[
 \mathbb{E}(T_j^2) \le C_1^2 \mathbb{E}(T_j)
\]
and thus
\[
 \mathbb{E}(X_j^2) \le C_1^2 R^{-n} \Delta_N - R^{-2n} \Delta_N^2 ~.
\]

Since the $X_j$ are selfadjoint, $X_j^2$ is positive semidefinite,
thus we have in fact proved
\[
 0 \le  \mathbb{E}(X_j^2) \le C_1^2 R^{-n} \Delta_N
\]
For positive semidefinite matrices, $A \le B$ implies $\| A \| \le
\| B \|$, hence
\[
 \sigma^2 = \left\| \sum_{j=1}^s \mathbb{E}(X_j^2) \right\| \le \left\| \sum_{j=1}^s  C_1 ^2 R^{-n} \Delta_N \right\|
 \le s C_1 ^2 R^{-n}
\]
\end{proof}

We can now formulate and prove a random sampling statement for $\mathcal{P}_N$.

\begin{thm} \label{thm:samp_PN}
 Let $(x_j)_{j \in \mathbb{N}}$ denote a sequence of independent and identically distributed random variables, uniformly distributed in $C_R$. Assume that $R \ge \sqrt[n]{C_1^2}$. Then, for all $\nu \ge 0$ and $s \in \mathbb{N}$:
 \[
  \mathbb{P} \left( \inf_{f \in \mathcal{P}_N, \| f \|_2 = 1} \frac{1}{s} \sum_{j=1}^s \left( |f(x_j)|^2- R^{-n} \| Q_R f \|_2^2 \right) \le R^{-n} \nu \right) \le N \exp \left( -\frac{\nu^2 s}{C_1^2 R^n (1+\nu/3)} \right)~.
 \]
\end{thm}
\begin{proof}
 Let $T_j$ be defined by (\ref{eqn:T_j}). Using the fact that computing expectations amounts to integration over $C_R$ (with respect to Lebesgue measure, normalized to one), in conjunction with $P_\Phi \psi_n = \psi_n$, one readily sees that
 \begin{eqnarray*}
 \left( \mathbb{E} (T_j) \right)_{k,\ell} & = & R^{-n} \int_{C_R} \psi_\ell (x) \overline{\psi_k(x)} dx \\
 & = &R^{-n} \langle Q_R \psi_\ell, \psi_k \rangle \\
 & = & R^{-n} \langle Q_R P_\Phi \psi_\ell, P_\Phi \psi_ k \rangle \\
 & = & R^{-n} \langle A_R \psi_\ell, \psi_k \rangle  \\
 & = & R^{-n} \lambda_{\ell} \delta_{\ell,k}
 \end{eqnarray*}
 and therefore
\[
\mathbb{E} (T_j)  =  R^{-n} \Delta_N~.
\]
 Furthermore, any unit vector $f \in \mathcal{P}_N$ is of the form $f = \sum_{n=1}^N c_n \psi_n$ with a unit vector $(c_n)_n$ of coefficients, and one obtains
 \begin{equation} \label{eqn:fxj}
  |f(x_j)|^2 = \langle c, T_j c \rangle~.
 \end{equation}
We thus have
 \begin{eqnarray*}
\lefteqn{  \inf_{f \in \mathcal{P}_N, \| f \|_2 = 1} \frac{1}{s} \sum_{j=1}^s \left( |f(x_j)|^2- R^{-n} \| Q_R f \|_2^2 \right)  = } \\
& = & \inf_{c \in \mathbb{C}^N, \| c \|_2 = 1 } \frac{1}{s} \sum_{j=1}^s \underbrace{\left(\langle c, T_j c \rangle - \langle c, \mathbb{E}(T_j) c \rangle \right)}_{= \langle c,X_j c \rangle}  \\
& = & \min \left\{ \rho~:~ \rho \mbox{ eigenvalue of } \frac{1}{s} \sum_{j=1}^s  X_j  \right\}~.
 \end{eqnarray*}
Now the statement of the theorem follows from the Bernstein
inequality for matrices formulated in Theorem \ref{bernstein}, with
proper constants provided by Lemma \ref{lem:exp_T_j}.
\end{proof}

\section{Proof of the main result}

It remains to transfer the random sampling statements from the spaces $\mathcal{P}_N$ to the set $V_{R,\delta}(\Phi)$. The following lemma is a first step in this direction,  by providing
a norm estimate for the projection onto $\mathcal{P}_N$,
for elements of $V_{R,\delta}(\Phi)$. It is an analog
of \cite[Lemma 5]{BaGro_12}.
\begin{lemma}
 \label{lem:B_vs_PN}
 Let $N \in \mathbb{N}$, and let $\gamma\in \mathbb{R}$ with $\lambda_N \ge \gamma \ge \lambda_{N+1}$. Let $E_N$ denote the orthogonal projection onto $\mathcal{P}_N$, and $F_N$ denote the projection onto the orthogonal complement. Then for all $f \in V_{R,\delta}(\Phi)$, we have
 \begin{eqnarray*}
  \| E_N f \|_2^2 & \ge &  \left( 1 - \frac{\delta}{1-\gamma} \right) \| f  \|_2^2~,
  \\ \| Q_R E_N f \|_2^2  & \ge & \gamma  \left( 1 - \frac{\delta}{1-\gamma} \right) \| f  \|_2^2~, \\
  \| F_N f \|_2^2 & \le  & \frac{\delta}{1-\gamma} \| f \|_2^2~.
 \end{eqnarray*}
 If $N=N(R) \not= 0$, these estimates simplify to
  \begin{eqnarray*}
  \| E_N f \|_2^2 & \ge &  (1 - 2 \delta ) \| f  \|_2^2~,
  \\ \| Q_R E_N f \|_2^2  & \ge &   \left( \frac{1}{2}- \delta  \right) \| f  \|_2^2~, \\
  \| F_N f \|_2^2 & \le &  2 \delta \| f \|_2^2~.
 \end{eqnarray*}

\end{lemma}
\begin{proof}
 Let $f \in V_{R,\delta}(\Phi)$, w.l.o.g. $\| f \|_2 = 1$. Since $f = P_\Phi f$, we obtain
 \[
  1 - \delta \le \| Q_R f \|_2^2 = \| Q_R P_\Phi f \|_2^2 = \langle Q_R P_\Phi f, Q_R P_\Phi f \rangle =
  \langle A_R f , f \rangle = \sum_{j} |\langle f, \psi_j \rangle|^2 \lambda_j ~.
 \]
 Let $c_j = \langle f, \psi_j \rangle$, and define
 \[
  A = \| E_N f \|_2^2 = \sum_{j=1}^N | c_j|^2,
 \] and $B = 1- A = \| F_N f \|_2^2$. Then $\sum_{j=N+1}^\infty |c_j|^2 \le \| F_N f \|_2^2 = 1-A$. Using  $\gamma \ge \lambda_{N+1} \ge \lambda_{N+2} > ...$ and $\lambda_j \le 1$, we find
 \begin{eqnarray*}
  A & = & \sum_{j=1}^N |\langle f, \psi_j \rangle|^2 \ge \sum_{j=1}^N |\langle f, \psi_j \rangle|^2 \lambda_j  \\ & = & \sum_{j=1}^\infty |c_j|^2 \lambda_j - \sum_{j=N+1}^\infty |c_j|^2 \lambda_j \\
  & \ge & 1- \delta - \gamma \left( \sum_{j=n+1}^\infty |c_j|^2 \right)   \\
  & \ge & 1-\delta- \gamma(1-A)~.
 \end{eqnarray*}
Solving this inequality for $A$ yields $A \ge
1-\frac{\delta}{1-\gamma}$, which implies $B \le
\frac{\delta}{1-\gamma}$. Finally, $\gamma \le \lambda_N$ yields $\|
Q_R E_N f \|_2^2 = \sum_{j=1}^N \lambda_j |c_j|^2 \ge \gamma A \ge
\gamma (  1-\frac{\delta}{1-\gamma} )$.

For $N=N(R) \not= 0$, we may pick $\gamma = 1/2$, which results in
the estimates given for this case.
\end{proof}

With the estimates from Lemma \ref{lem:B_vs_PN}, the proof of the next lemma is a verbatim adaptation of the argument showing \cite[Lemma 7]{BaGro_12}, and therefore omitted.
\begin{lemma} \label{lem:geometric} Let $N \in \mathbb{N}$ and $\lambda_N \ge \gamma \ge \lambda_{N+1}$.
 Let $\{ x_j: j=1,\ldots, s \} \subset C_R$ have covering index $N_0$, and assume that the inequality
 \[
  \frac{1}{s} \sum_{j=1}^s \left( |p(x_j)|^2- R^{-n} \| Q_R p \|_2^2\right) \ge -\nu R^{-n} \| f  \|_2^2
 \] holds for all $p \in \mathcal{P}_N$. Then the inequality
 \begin{equation}
  \sum_{j=1}^s |f(x_j)|^2 \ge A\| f \|_2^2
 \end{equation}
 holds for all $f \in V_{R,\delta}(\Phi)$, with the constant
 \[
  A = \frac{s}{R^n} \left( \gamma - \frac{\gamma \delta}{1-\gamma} - \nu \right) - 2 C_2  N_0 \frac{\delta}{1-\gamma}~.
 \]
\end{lemma}

A further ingredient is the following tail estimate for the covering
number of a random set.
\begin{lemma}\label{covernumber}\cite[Lemma 8]{BaGro_12} Suppose $R\geq 2$
and $\{x_j : j=1,\cdots , s\}$ are independent and identically
distributed random variables that are uniformly distributed over
$C_R$. Let $a> R^{-n}$. Then \[\mathbb{P}(N_0 > as)\leq (R+2)^n
exp\big(-s(a \log{(aR^n)}-(a-R^{-n}))\big),\] where $N_0 =\max_{k\in
\mathbb{Z}^n}{\rm card}( \{x_j\}\cap(k+[-1/2, 1/2]^n)).$
\end{lemma}

\begin{thm} \label{thm:main_aux}
 Let $(x_j)_{j \in \mathbb{N}}$ be a sequence of independent  random variables,
 each uniformly distributed in $C_R$. Assume that $R \ge \max(1,\sqrt[\alpha]{2 C_3},\sqrt[n]{C_1^2})$, and furthermore
 \[
  \delta < \frac{1}{2(1+12 C_2 )}~, \nu < \frac{1}{2} - \delta (1+12 C_2 )~.
 \] Then, for any $s \in \mathbb{N}$,
 \[
  A = \frac{s}{R^n} \left( \frac{1}{2} - \delta - \nu - 12 \delta C_2  \right)
 \] is strictly positive, and the sampling estimate
 \begin{equation} \label{eqn:samp_ineq}
  A \| f \|_2^2 \le \sum_{j=1}^s |f(x_j)|^2 \le s  C_1^2 \| f  \|_2^2~,~ \forall f \in V_{R,\delta}(\Phi)
 \end{equation}
holds with probability at least
\begin{equation} \label{eqn:prob_est_aux}
 1- R^{n^2/\alpha'} \beta^n \exp \left( -\frac{\nu^2 s}{C_1^2 R^n (1+\nu/3)} \right) - (R+2)^n \exp \left( - \frac{s}{R^n} (3 \log 3 -2) \right)~.
\end{equation} Here  $\beta$ is defined by (\ref{eqn:beta}).
\end{thm}
\begin{proof}
Define the random variable $N_0$ as the covering index of $x_1,\ldots,x_s$.
Fix $N = N(R)$, and consider the events
 \[
  V_1 =  \left\{ \inf_{f \in \mathcal{P}_N, \| f \|_2 = 1} \frac{1}{s} \sum_{j=1}^s \left( |f(x_j)|^2 - R^{-n} \| Q_R f \|_2^2 \right) \le - \nu R^{-n} \right\}~
 \] and
 \[
  V_2 = \left\{ N_0 \ge 3 R^{-n} \right\}~.
 \]
By Lemma \ref{lem:geometric}, we have for all $(x_1,\ldots,x_s)$ in the complement of $V_1 \cup V_2$ and all $f \in V_{R,\delta}(\Phi)$,
\[
 \frac{1}{s} \sum_{j=1}^s |f(x_j)|^2 \ge A \| f \|_2^2~,
\] where we used $\gamma = 1/2$ (due to our choice of $N=N(R)$) to simplify $A$ to the constant occurring in (\ref{eqn:samp_ineq}).

Theorem \ref{thm:samp_PN} combined with Lemma \ref{lem:decay_eigenvals} yields that $V_1^c$ occurs with probability at most
\[
 \beta^n R^{n^2/\alpha'} \exp \left( -\frac{\nu^2 s}{C_1 ^2 R^n (1+\nu/3)} \right)~.
\] Furthermore, Lemma \ref{covernumber} yields that $V_2^c$ occurs with probability at most
\[
 (R+2)^n \exp \left( - s R^{-n} (3 \log 3 - 2) \right)~.
\]
Thus the lower estimate in (\ref{eqn:samp_ineq}) occurs at least with the probability given in (\ref{eqn:prob_est_aux}), whereas the upper estimate follows from the definition of $C_1$.
\end{proof}

{\em Proof of Theorem \ref{thm:main}.}

The requirements on the various quantities guarantee the applicability of Theorem \ref{thm:main_aux}.
Furthermore, for $R\ge 1$, we find that by assumptions on $R$ and $\nu$,
\begin{eqnarray} \nonumber
 \lefteqn{\frac{(R+2)^n \exp\left(-\frac{s}{R^n} (3 \log 3 -2)\right)}{\beta^n R^{n^2/\alpha'} \exp \left(-\frac{s \nu^2}{R^n C_1^2 (1+\nu/3)} \right)}} \\ & \le & \underbrace{\frac{3^n R^{n-n^2 /\alpha'}}{\beta^n}}_{\le 1,\ since\ \beta>3} \underbrace{ \exp\left( - \frac{s}{R^n} \left(  3 \log 3 -2 - \frac{\nu^2}{ C_1^2 (1+\nu/3)} \right) \right) }_{\le 1} \nonumber \\
 & \le 1 ~. \label{eqn:comp_subterms}
\end{eqnarray}

Hence, as soon as
\[
s \ge R^n  \frac{1+\nu/3}{\nu^2} \log \frac{2 \beta^n
R^{n^2/\alpha'}}{\epsilon}~,
\] the first term subtracted in (\ref{eqn:prob_est_aux}) is $\le \epsilon/2$, and greater or equal to the second term, by (\ref{eqn:comp_subterms}). The theorem is proved.
\hfill $\Box$

%\bibliography{FuehrXian_relevant_sampling.bib}
%\bibliographystyle{plain}

\end{document}